%% file: genHoeffding.tex
\pdfoutput=1
\documentclass[preprint,svgnames,dvipsnames]{elsarticle}

\input{macros}

\begin{document}
\begin{frontmatter}

\title{Hoeffding decomposition of black-box models with dependent inputs}


\author[a,b,c,e]{Marouane Il Idrissi}
\author[a,b,d]{Nicolas Bousquet}
\author[c]{Fabrice Gamboa}
\author[a,b,c]{Bertrand Iooss}
\author[c]{Jean-Michel Loubes}

\address[a]{EDF Lab Chatou, 6 Quai Watier, 78401 Chatou, France}
\address[b]{SINCLAIR AI Laboratory, Saclay, France}
\address[c]{Institut de Mathématiques de Toulouse, 31062 Toulouse, France}
\address[d]{Sorbonne Université, LPSM, 4 place Jussieu, Paris, France}
\address[e]{Corresponding Author - Email: ilidrissi.m@gmail.com}

\begin{abstract}
Performing an additive decomposition of arbitrary functions of random elements is paramount for global sensitivity analysis and, therefore, the interpretation of black-box models. The well-known seminal work of Hoeffding characterized the summands in such a decomposition in the particular case of mutually independent inputs. Going beyond the framework of independent inputs has been an ongoing challenge in the literature.  Existing solutions have so far required constraining assumptions or suffer from a lack of interpretability.  In this paper, we generalize Hoeffding's decomposition for dependent inputs under very mild conditions. For that purpose, we propose a novel framework to handle dependencies based on probability theory, functional analysis, and combinatorics. It allows for characterizing two reasonable assumptions on the dependence structure of the inputs: non-perfect functional dependence and non-degenerate stochastic dependence. We then show that any square-integrable, real-valued function of random elements respecting these two assumptions can be uniquely additively decomposed and offer a characterization of the summands using oblique projections. We then introduce and discuss the theoretical properties and practical benefits of the sensitivity indices that ensue from this decomposition. Finally, the decomposition is analytically illustrated on bivariate functions of Bernoulli inputs.
\end{abstract}


\end{frontmatter}
\section{Introduction}

For a positive integer $d$, let $X=(X_1, \dots, X_d)$, be an $E$-valued random element, where $E$ is a cartesian product of Polish spaces, and let $G : E \rightarrow \R$. \cite{Hoeffding1948} showed that, as long as $\E{G(X)^2} < \infty$ and that the vector $X$ is comprised of mutually independent elements, then $G(X)$ can be uniquely decomposed as
\begin{equation}
    G(X) = \sum_{A \in \pset{D}} G_A(X_A), ~\text{such that for any}~ A \in \pset{D},~G_A(X_A) = \sum_{B \in \pset{A}} (-1)^{|A|-|B|}\Ebb_B\lrbra{G(X)},
    \label{eq:HDMR}
\end{equation}
where $D = \lrcubra{1,\dots,d}$, $\pset{D}$ is the power-set of $D$, and for $A\in \pset{D},~ \Ebb_A\lrbra{\cdot} = \E{\cdot \mid X_A}$ is the conditional expectation operator. This decomposition, also known as a \emph{high-dimensional model representations} (HDMR) \citep{Rabitz1999}. is at the cornerstone of many fields such as sensitivity analysis (SA) \citep{DaVeiga2021}, explainability in artificial intelligence (XAI) \citep{Barredo2020}, and more broadly, the uncertainty quantification (UQ) of critical systems for industrial practices \citep{deRocquigny2008, Ghanem2017}.

In these fields, one of the main and most pressing challenges is to deal with dependent inputs \citep{Razavi2021}, as the wide majority of use-cases do not exhibit independent inputs. However, the most popular methods often assume mutual independence \citep{Sobol2001, Lundberg2017}, either for the simplicity of the resulting estimation schemes or due to the lack of a proper framework. Always assuming mutual independence in practice can be seen as expedient and can lead to improper insights \citep{Hart2018}.

To adapt the decomposition in \Eqref{eq:HDMR} for dependent inputs, many approaches have been proposed in the literature. Notably, \cite{Hart2018} proposed an approximation theoretic framework and provides valuable tools for importance quantification, but their approach lacks a clear and intuitive explanation of the quantities being estimated. On the one hand, in \cite{Chastaing2012}, the authors approached the problem differently and brought forward an intuitive view on the subject, but under significantly limiting assumptions on the probabilistic structure of the inputs. On the other hand, \cite{Hooker2007} and \cite{Kuo2009} proposed a projection-based approach under constraints derived from desirability criteria. \cite{Peccati2004} proposed a generalization under input exchangeability. However, these approaches do not offer a complete picture of how uncertainties behave under input dependencies. Other approaches rely on transforming the dependent inputs to achieve mutual independence using, \eg Nataf or Rosenblatt transforms \citep{Lebrun2009a, Lebrun2009b}. For instance, \cite{Mara2015} offers meaningful insights into the interactions between the inputs induced by the model $G$ and the dependence, but in a restrictive setting.

To fill these gaps, we propose to tackle this generalization through the prism of an original framework at the cornerstone of probability theory, functional analysis, and abstract algebra. By viewing random variables as measurable functions, we prove that a unique decomposition such as in \eqref{eq:HDMR}, for square-integrable black-box outputs $G(X)$, is indeed possible under two fairly reasonable assumptions on the inputs:
\begin{enumerate}
    \item Non-perfect functional dependence;
    \item Non-degenerate stochastic dependence.
\end{enumerate}

More precisely, denote by $\sigma_X$ the $\sigma$-algebra generated by $X$, and $\hLs{X}$ the Lebesgue space of square-integrable $\sigma_X$-measurable real-valued random variables. We effectively show that under the previously stated assumptions, the direct-sum decomposition
$$\hLs{X} = \bigoplus_{A \in \pset{D}} V_A,$$
is valid, where $V_A$ are linear vector subspaces of functions of $X_A$, which are completely characterized. To our knowledge, it offers a novel approach to multivariate dependencies, relying on geometric considerations. We also show that Hoeffding's classical decomposition is a special case of our proposed framework. 

Furthermore, we introduce novel sensitivity indices based on this decomposition. In particular, it allows defining theoretically justified model evaluation decompositions. In addition, we also propose four indices for quantifying the input importance from the resulting variance decomposition of $G(X)$. They offer a clear and intuitive way to measure the effects of interactions and the effects of dependencies.

This article is organized as follows. \Secref{sec:preli} is dedicated to introducing the overall framework, notations, and preliminary results. \Secref{sec:proof} is dedicated to proving the main result of this paper. \Secref{sec:Obs} is dedicated to discussing some of the consequences of this novel decomposition. In \Secref{sec:decomps}, we apply the model's decomposition for XAI and global SA purposes. \Secref{sec:illustration} is dedicated to illustrating our result in the particular case of a model with two Bernoulli inputs. Finally, \Secref{sec:conclu} discusses the challenges for a broad acceptance of the proposed method in practice, as well as some motivating perspectives.

\section{Framework and preliminaries}\label{sec:preli}

In the remainder of this document, the following notations are adopted. $\subset$ indicates a proper (strict) inclusion, while $\subseteq$ indicates a possible equality between sets. The positive integer $d$ denotes the number of inputs (\ie covariates), and $D \eqdef \{1,\dots, d\}$. ``Random element'' is used if the domain of measurable functions is not necessarily $\R$ or $\R^d$. The term ``random variables'' is exclusive to \emph{real-valued random elements}, and ``random vector'' is exclusive to a vector of random variables. Independence between random elements is denoted using $\indep$. For a measurable space $(E, \calE)$ and $B \subset \calE$, we denote by $\sigma\lrbra{B}$ the smallest $\sigma$-algebra containing $B$. For a finite set, \eg $D$, denote by $\pset{D}$ its power-set (\ie the set of subsets of $D$, including $D$ and $\emptyset$), and for any $A \in \pset{D}$, denote $\pset{-A} = \pset{A} \setminus \{A\}$ (\ie the power-set of $A$ without $A$). Depending on the context, when dealing with a set $A$, $\absval{A}$ denotes its cardinality (\ie the number of elements in $A$), while for a scalar $c \in \R$, $\absval{c}$ denotes its absolute value.

\subsection{Random inputs}
Let $(\Omega, \calF, \prob)$ be a probability space, and let $(E_1, \calE_1), \dots, (E_d, \calE_d)$ be a collection of standard Borel measurable spaces. For every $A \in \pset{D}, A \neq \emptyset$, denote:
$$E_A \eqdef \bigtimes_{i \in A} E_i, \quad \calE_A \eqdef \bigotimes_{i \in A} \calE_i,$$
where $\times$ denotes the Cartesian product between sets and $\otimes$ denotes the product of $\sigma$-algebras. Denote $E \eqdef E_D$ and $\calE \eqdef \calE_D$.

\emph{Random inputs} refer to $E$-valued, ($\calF$)-measurable mappings, and are denoted $X = (X_1, \dots, X_d)$. For random inputs $X$, the \emph{subsets of the inputs} refer to, for every $A \subset D$, the $E_A$-valued random elements defined as $X_A \eqdef (X_i)_{i \in A}$.

Denote by $\sigma_{\emptyset}$ the $\prob$-trivial $\sigma$-algebra, \ie $\sigma_{\emptyset} \eqdef \sigma \lrbra{\lrcubra{A \in \calF : \Prob{A} =0}}$, which is the smallest $\sigma$-algebra containing every null sets of $\calF$ \wrt $\prob$. Denote by $\sigma_X$ the $\sigma$-algebra generated by $X$. For every $A \in \pset{D}, A \not \in \lrcubra{\emptyset, D}$, let $\sigma_A$ be the $\sigma$-algebra generated by the subset of inputs $X_A$.

Throughout this paper, it is assumed that:
    \begin{enumerate}
        \item For every $i \in D$, $\sigma_{\emptyset} \subset \sigma_i$, \ie the $\prob$-trivial $\sigma$-algebra is strictly contained in the $\sigma$-algebras generated by individual inputs;
        \item For every $A,B \in \pset{D}$ such that $B \subset A$, $\sigma_B \subset \sigma_A$, \ie the $\sigma$-algebra generated by a subset of inputs is necessarily strictly contained in the $\sigma$-algebras generated by a bigger subset of inputs;
    \end{enumerate}
    
\begin{rmk}
    While (1.) is standard in many probabilistic theoretical frameworks (\seg \cite{Sidak1957}), (2.) is less standard but remains reasonable since it avoids redundancy between subsets of inputs.
\end{rmk}

\subsection{Space of square-integrable outputs and its subspaces}
Denote $\hL(\calF)$ the Lebesgue space of $\R$-valued square-integrable $\calF$-measurable functions on $\lrpar{\Omega, \calF, \prob}$. For any sub $\sigma$-algebra $\calB \subseteq \calF$, denote $\hL\lrpar{\calB}$ the subspace of $\hL(\calF)$ containing $\calB$-measurable functions. $\hL_0\lrpar{\calB}$ denotes the subspace of centered elements of $\hL\lrpar{\calB}$. Let us recall classical result \citep[Theorem 2]{Sidak1957}:
\begin{thm}
    For two sub $\sigma$-algebras $\calB_1$ and $\calB_2$ of $\calF$, the following assertions hold.
    \begin{enumerate}
        \item If $\calB_1 \subseteq \calB_2$, then $\hL\lrpar{\calB_1} \subseteq \hL\lrpar{\calB_2}$.
        \item $\hL\lrpar{\calB_1} \cap \hL\lrpar{\calB_2} = \hL \lrpar{\calB_1 \cap \calB_2}$.
    \end{enumerate}
    \label{thm:sidak}
\end{thm}
Thanks to the Doob-Dynkin Lemma, it is obvious that the Lebesgue space $\hLs{X}$ only contains random variables that are functions of $X$. Additionally, notice that $\hLs{\emptyset}$ only contains constant a.s. functions \citep[Lemma 4.9]{Kallenberg2021}. Moreover, for every $A \in \pset{D}$, the Lebesgue spaces $\hLs{A}$ are Hilbert spaces \citep[Theorem 9.4.1]{Malliavin1995} and are nested following the partial order of $\pset{D}$. In the following, for every $A \in \pset{D}$, the closed subspace $\hLs{A}$ of $\hLs{X}$ is referred to as \emph{the Lebesgue space generated by $X_A$}.

\subsection{Angles between subspaces}

We introduce two angles between closed subspaces of a Hilbert space, which are paramount in the upcoming developments: Dixmier's angle and Friedrichs' angle.

Dixmier's angle \citep{Dixmier1949} is the \emph{minimal angle} between two closed subspaces of a Hilbert space, and its cosine is defined as:
\begin{dfi}[Dixmier's angle]
    Let $H$ and $K$ be closed subspaces of a Hilbert space $\calH$ with inner product $\iprod{\cdot,\cdot}$ and norm $\norm{\cdot}$. The cosine of Dixmier's angle is defined as
    $$\cz{H,K} \eqdef \sup \lrcubra{\absval{\iprod{x,y}} : x\in H, \norm{x} \leq 1, \quad y\in K, \norm{y}\leq 1}.$$
\end{dfi}
In probability theory, this angle is analogous to the \emph{maximal correlation} \citep{Gebelein1941}, and has been extensively studied in the literature (\seg \cite{Renyi1959, Koyak1987, Doukhan1994, Dembo2001, Dauxois1997}). Additionally, note that a maximal correlation of $0$ between two random elements is equivalent to their independence \wrt $\prob$ (\see \cite[Chapter 3]{Malliavin1995}).

Friedrichs' angle \citep{Friedrichs1937} differs from Dixmier's angle in one way: the intersection of the two subspaces is disregarded. Its cosine is defined as:
\begin{dfi}[Friedrichs' angle]
    Let $H$ and $K$ be closed subspaces of a Hilbert space $\calH$ with inner product $\iprod{\cdot , \cdot}$ and norm $\norm{\cdot}$. The cosine of Friedrichs' angle is defined as
    $$\cd{H,K} \eqdef \sup \lrcubra{\absval{\iprod{x,y}} : 
    \begin{cases}
        x\in H \cap \lrpar{H \cap K}^\perp, \norm{x} \leq 1 \\
        y\in K \cap \lrpar{H \cap K}^\perp, \norm{y}\leq 1
    \end{cases}},$$
    where the orthogonal complement is taken \wrt to $\calH$.
    \label{def:friedAng}
\end{dfi}
In probability theory, this quantity is analogous to the \emph{maximal partial (or relative) correlation} \citep{Bryc1984, Bryc1996, Dauxois2004}. It offers a sufficient and necessary condition for the commutativity of conditional expectations \wrt $\sigma$-algebras (\see \cite{Kallenberg2021}, Theorems 8.13 and 8.14). 

\begin{rmk}
In the following, any reference to Friedrichs' or Dixmier's angle refers to the \emph{cosine of the angle} (taking values in $[0,1]$) instead of the angle itself (which takes its values in $[0, \pi/2]$).
\end{rmk}

In functional analysis, these two angles are often used to assess the closedness of sums of closed subspaces of a Hilbert space. The interested reader is referred to the work of \cite{Deutsch1995} for a more complete overview of their properties. We recall the relevant properties for the upcoming developments. 

\begin{lme}
    Let $H,K$ be closed subspaces of a Hilbert space $\calH$. 
    \begin{enumerate}
        \item $0 \leq \cd{H,K} \leq \cz{H,K} \leq 1$;
        \item For any $x \in H$ and any $y \in K$:
            \begin{equation}
                \absval{\iprod{x,y}} \leq \cz{H,K} \norm{x}\norm{y};
                \label{eq:sharpCS}
            \end{equation}
        \item For a proper closed subspace $\widetilde{H} \subset H$, $\cz{\widetilde{H}, K} \leq \cz{H,K}$;
        \item $\cz{H,K} <1$ $\iff$ $H \cap K = \{0\}$ and $H+K$ is closed in $\calH$;
        \item $\cd{H,K} <1 $ $\iff$ $H+K$ is closed in $\calH$;
        \item $\cd{H,K} = \cz{H \cap \lrpar{H \cap K}^\perp, K} = \cz{H, K\cap \lrpar{H \cap K}^\perp}$;
        \item If $H\cap K = \{0\}$, then $\cd{H,K} = \cz{H,K}$.
    \end{enumerate}
   \label{lme:relaAngles}
\end{lme}

\subsection{Feshchenko matrix}
As illustrated above, the maximal and partial correlation are good candidates to control the dependence structure of random elements. They can be understood as a generalization to random elements of the correlation and partial correlation of random variables. Hence, they offer a natural avenue for generalizing to the notion of \emph{covariance and precision matrices}.

It is well-known that precision matrices (\ie inverses of covariance matrices) can be written using partial correlations (\seg \cite[p.129]{Lauritzen1996}). Generalized precision matrices are not new. In the study of \emph{graphical models}, generalized covariance and precision matrices have been used to study $\sigma$-algebras following a graph structure \citep{Loh2013}. 

We propose a novel approach, by substituting the partial correlation with Friedrichs' angle to define a generalization of precision matrices to random elements. It is defined as:
\begin{dfi}[Maximal coalitional precision matrix]
    Let $X = (X_1, \dots, X_d)$ be random inputs. The maximal coalitional precision matrix of $X$ is the $\lrpar{2^d \times 2^d}$ symmetric, set-indexed matrix $\Delta$, defined entry-wise, for any $A,B \in \pset{D}$, by
    $$\Delta(A,B) = \begin{cases}
        1 & \text{ if } A = B;\\
        -\cd{\hLs{A}, \hLs{B}} & \text{otherwise.}
    \end{cases}$$
    Moreover denote by $\DeltA{A}$ the principal $\lrpar{2^{|A|}-1 \times 2^{|A|}-1}$ submatrix of $\Delta$ relative to the proper subsets of $A \in \pset{D}$.
    \label{def:MCPM}
\end{dfi}
In functional analysis, an analogous matrix is used to derive a sufficient condition for sums of more than two closed subspaces of an abstract Hilbert space to be closed. The interested reader can refer to the pioneering works of Ivan Feshchenko \citep{Feshchenko2012, Feshchenko2020}. This matrix is primarily used for that purpose in the proof of our main theorem. In the following, the matrix $\Delta$ defined in \defref{def:MCPM} is referred to as the \emph{Feshchenko matrix}.

One particular and interesting aspect of the Feshchenko matrix of random inputs is that \emph{it is equal to the identity if and only if the inputs are mutually independent} (\see \Secref{sec:mutIndep}).

\subsection{Closure, complements, and projections}

A closed proper subspace $H$ of a Hilbert space $\calH$ is always \emph{complemented}, \ie there exist some subspace $K$ of $\calH$ such that $\calH$ admits the \emph{direct-sum decomposition} (\seg \cite[Definition 1.40 and Proposition 1.44]{Axler2015}) $\calH = H \oplus K$. The most well-known complement is probably the \emph{orthogonal complement} $H^\perp$ (\seg \cite[Theorem 12.4]{Rudin1996}) of $H$ in $\calH$. Let $A \in \pset{D}$ and let $B \subset A$, and for a subspace $H \subset \hLs{B}$, denote by $H^\perpa{A}$ the \emph{orthogonal complement of $H$ in $\hLs{A}$}.
\begin{lme} 
    Let $A,B \in \pset{D}$, such that $B \subseteq A$, and let $H$ be a subspace of $\hLs{B}$. Then
    $$H^\perpa{B} \subseteq H^\perpa{A}.$$
    \label{lme:perpas}
\end{lme}
\begin{proof}[Proof of \lmeref{lme:perpas}]
    From \thmref{thm:sidak}, one has that $\hLs{B} \subseteq \hLs{A}$. The result is a direct consequence of the definition of orthogonal complements.
\end{proof}
\begin{rmk}
    It is important to note that the orthogonal complement is not necessarily the only complement of a closed proper subspace of a Hilbert space. Other (not necessarily orthogonal) complements can be found to respect the direct-sum decomposition of the ambient space.
\end{rmk}
For two Hilbert spaces $\lrpar{\calM_1, \norm{\cdot}_1}$ and $\lrpar{\calM_2, \norm{\cdot}_2}$, and a linear operator $T : \calM_1 \rightarrow \calM_2$ denote the \emph{range} of $T$ by $\Ran{T} \eqdef \lrcubra{T(x) : x \in \calM_1} \subseteq \calM_2$ and its \emph{nullspace} by $\Ker{T} \eqdef \lrcubra{x \in \calM_1 : T(x) = 0} \subseteq \calM_1$.

Let $\calH$ be a Hilbert space and $P : \calH \rightarrow \calH$ be a bounded linear operator. If $P$ is also idempotent (\ie $P \circ P = P$), then $\calH$ admits the direct sum decomposition $\calH = \Ran{P} \oplus \Ker{P}$ (\see \cite[Proposition 3.2]{Conway2007}), and $P$ is called the \emph{oblique projector} on $\Ran{P}$ parallel to $\Ker{P}$, defined as
\begin{align*}
    P : \calH = \Ran{P} \oplus \Ker{P} & \rightarrow \calH \\
    x = x_R + x_K & \mapsto x_R
\end{align*}
where $x_R \in \Ran{P}$ and $x_K \in \Ker{P}$. The operator $I-P$ is the oblique projection on $\Ker{P}$, parallel to $\Ran{P}$. If there are two closed subspaces $M$ and $N$ of a Hilbert space $\calH$ such that $\calH = M \oplus N$, then there exists a continuous idempotent operator (\ie an oblique projector) $P$ with range $\Ran{P} = M$ and $\Ker{P}=N$ (\see \cite[Theorem 7.90]{Galantai2004}). In this case, $P$ is known as the \emph{canonical projector} (\wrt to the direct sum decomposition $\calH = M \oplus N$). If $\Ker{P} = \Ran{P}^\perp$, then the projection is said to be \emph{orthogonal}, which is equivalent to $P$ being self-adjoint (\see \cite[Theorem 7.71]{Galantai2004}). The direct-sum decomposition, in this case, is justified thanks to Hilbert's projection theorem. 

It is well known that, for every $A \in \pset{D}$, the conditional expectation operators \wrt $\sigma_A$, denoted $\condE{A}{\cdot}$ are the \emph{orthogonal projectors} of elements of $\hLs{X}$ onto $\hLs{A}$ (parallel to $\hLs{A}^\perp$).

\section{Generalized Hoeffding decomposition}\label{sec:proof}

As explained in the introduction, our main result consists of showing that $\hLs{X} = \bigoplus_{A \in \pset{D}} V_A$, where each $V_A$ is a subspace of $\hLs{A}$. The generalization of Hoeffding's decomposition directly ensues from such a direct-sum decomposition. In this section, we show that such direct-sum decompositions are achievable under two reasonable assumptions on the inputs, which are less restrictive than mutual independence. These two assumptions are introduced and discussed before proving our main result.

\subsection{Non-perfect functional dependence}
Our first assumption can be formulated as a condition on the intersection of the $\sigma$-algebras generated by the subsets of inputs.
\begin{assu}[Non-perfect functional dependence]
    For any $A, B \in \pset{D}$, $\sigma_A \cap \sigma_B = \sigma_{A \cap B}$.
    \label{assu:a1}
\end{assu}
The name ``non-perfect functional dependence'' is justified by the following proposition.
\begin{prop}
    Let $X=(X_1, \dots, X_d)$ be inputs, and suppose that \Assuref{assu:a1} holds. Then, for any $A,B \in \pset{D}$ such that $A \cap B \not \in \{A,B\}$ (\ie the sets cannot be subsets of each other), there is no mapping $T : E_A \rightarrow E_B$ such that $X_B = T(X_A)$ a.s.
    \label{prop:funcDep}
\end{prop}
\begin{proof}[Proof of \propref{prop:funcDep}]
    Suppose that there exists a mapping $T : E_A \rightarrow E_B$ such that $X_B = T(X_A)$ a.s. Then, one has that $\sigma_B \subseteq \sigma_A$, which in turn implies that $\sigma_A \cap \sigma_B = \sigma_{B}$. Notice that necessarily $A \cap B \subset B$ and in the present framework $\sigma_{A \cap B} \subset \sigma_B$. Thus $\sigma_A \cap \sigma_B$ is necessarily different than $\sigma_{A \cap B}$, and thus \Assuref{assu:a1} cannot hold. The result follows by taking the opposite implication.
\end{proof}

\subsection{Non-degenerate stochastic dependence}
The second assumption restricts the distribution of $X$ through a condition on the \emph{inner product of the Lebesgue space $\hLs{X}$}. More precisely, it amounts to control the angles between the subspaces $\hLs{A}$, $A\in \pset{D}$ using the Feshchenko matrix of the inputs $X$.
\begin{assu}[Non-degenerate stochastic dependence]
    The Feshchenko matrix $\Delta$ of $X$ is positive definite.
    \label{assu:a2}
\end{assu}
Since $\Delta$ can be seen as a generalized precision matrix, this assumption is relatively reasonable since standard precision matrices are often assumed to be positive definite to ensure non-degeneracy of the probabilistic structure. Moreover, having a definite positive Feshchenko matrix entails that the maximal partial correlation between $X_A$ and $X_B$ is strictly less than $1$ (\ie the angle itself must be greater than zero), which motivates its name.
\begin{prop}
    Suppose that \Assuref{assu:a2} hold. Then, for any $A,B \in \pset{D}$ such that $A \neq B$,
    $$\cd{\hLs{A}, \hLs{B}} <1.$$
    \label{prop:controlCD}
\end{prop}
\begin{proof}[Proof of \propref{prop:controlCD}]
    Suppose that \Assuref{assu:a2} hold. Then, in particular, the principal submatrix of $\Delta$
    $$\begin{pmatrix}
        1 & -\cd{\hLs{A}, \hLs{B}} \\
        -\cd{\hLs{A}, \hLs{B}} & 1
    \end{pmatrix}$$
    is positive definite as well, and thus,
    $$2 - 2\cd{\hLs{A}, \hLs{B}}>0 \quad \iff \quad \cd{\hLs{A}, \hLs{B}} <1.$$
\end{proof}

\subsection{Proof of the main result}
Our main result can be stated as follows:
\begin{thm}[Direct-sum decomposition of generated Lebesgue spaces]
    For every $A \in \pset{D}$, let $V_{\emptyset} = \hLs{\emptyset}$ and for every $B \in \pset{A}$, let 
    $$V_B = \lrbra{\bigplus_{C \in \pset{-B}} V_C}^\perpa{B}.$$
    If Assumptions~\ref{assu:a1} and \ref{assu:a2} hold, then for every $A \in \pset{D},$ one has that
    $$\hLs{A} = \bigoplus_{B \in \pset{A}} V_B.$$
    \label{thm:thm}
\end{thm}
If \thmref{thm:thm} is proven to be true, the HDMR of any element of $\hLs{X}$ follows directly (it is immediate by definition of direct-sums).
\begin{coro}[Orthocanonical decomposition]
    Let $X = (X_1, \dots, X_d)$ be random inputs. Suppose that Assumptions~\ref{assu:a1} and \ref{assu:a2} hold. Then, for any $G : E \rightarrow \R$ such that $G(X) \in \hLs{X}$, $G(X)$ can be uniquely decomposed as
    $$G(X) = \sum_{A \in \pset{D}} G_A(X_A),$$
    where each $G_A(X_A) \in V_A$.
    \label{coro:uniqDecomp}
\end{coro}
In order to prove \thmref{thm:thm}, two preliminary results are required.
\begin{lme}
    Let $A \in \pset{D}$, and let $B,C \in \pset{-A}$ be non-empty proper subsets of $A$ such that $B \neq C$. Let $V_B, V_C$ be a closed subspace of $\hLs{B}$ and $\hLs{C}$ respectively. Suppose that \Assuref{assu:a1} holds, and that
    $$V_B \subseteq \lrbra{\hLs{B \cap C}}^\perp, ~ \text{and}~ V_C \subseteq \lrbra{\hLs{B \cap C}}^\perp.$$
    Then,
    $$\cz{V_B, V_C} \leq \cd{\hLs{B}, \hLs{C}}.$$
    \label{lme:step1}
\end{lme}
\begin{proof}[Proof of \lmeref{lme:step1}]
    First, recall that, if \Assuref{assu:a1} holds and thanks to \thmref{thm:sidak}
    $$\hLs{B} \cap \hLs{C} = \hL\lrpar{\sigma_B \cap \sigma_C}= \hLs{B \cap C}.$$
    Then, notice that since
    $$V_B \subseteq \hLs{B} \cap \lrbra{\hLs{B \cap C}}^\perp, \quad \text{ and } V_C \subseteq \hLs{C} \cap \lrbra{\hLs{B \cap C}}^\perp,$$
    one has that
    \begin{align*}
        \cz{V_B, V_C} &= \cz{\hLs{B} \cap V_B, \hLs{C} \cap V_C} \\
        &\leq \cz{\hLs{B} \cap \lrbra{\hLs{B \cap C}}^\perp, \hLs{C} \cap \lrbra{\hLs{B \cap C}}^\perp}.
    \end{align*}
    Hence, if \Assuref{assu:a1} is assumed
    \begin{align*}
        \cz{V_B, V_C} &\leq \cz{\hLs{B} \cap \lrbra{\hLs{B} \cap \hLs{C}}^\perp, \hLs{C} \cap \lrbra{\hLs{B} \cap \hLs{C}}^\perp} \\
        &= \cd{\hLs{B}, \hLs{C}}
    \end{align*}
    where the last equality is achieved using \lmeref{lme:relaAngles}.
\end{proof}

\begin{lme}
    Let $A \in \pset{D}$, and let $\lrpar{V_B}_{B \in \pset{A}, B \neq A}$ be a collection of closed subspaces of $\hLs{A}$ such that, $\forall B,C \in \pset{-A}, B \neq C$,
    $$\cz{V_B, V_C} \leq \cd{\hLs{B}, \hLs{C}}.$$
    Then, under \Assuref{assu:a2}, there exist a $\rho > 0$ such that, for any $\sum_{A \in \pset{-A}} Y_A \in \bigplus_{B \in \pset{-A}} V_B$
    $$\sqrt{\E{ \lrpar{\sum_{B \in \pset{-A}}Y_A}^2}} \geq \rho \sum_{B \in \pset{-A}} \sqrt{\E{Y_A^2}},$$
    and additionally, the sum of subspaces $\bigplus_{B \in \pset{-A}} V_B$ is closed in $\hLs{A}$.
    \label{lme:step2}
\end{lme}
\begin{proof}[Proof of \lmeref{lme:step2}]
    Let $H_A = \bigoplus_{B \in \pset{A} : B \neq A} V_A$ be the Hilbert space external direct-sum (\see \cite[Definition 6.4]{Conway2007}) of the collection of closed (and thus Hilbert) subspaces $\lrpar{V_B}_{B \in \pset{A}, B \neq A}$. Let $T_A$ be the operator defined as
    \begin{align*}
        T_A : H_A & \rightarrow \hLs{A} \\
        Y = \lrpar{Y_B}_{B \in \pset{-A}} &\mapsto \sum_{B \in \pset{-A}} Y_B
    \end{align*}
    and notice that 
    $$\Ran{T_A} = \bigplus_{B \in \pset{-A}} V_B \subseteq \hLs{A}.$$
    One then has that
    \begin{align*}
        \E{\lrpar{\sum_{B \in \pset{-A}} Y_B}^2} &= \sum_{B \in \pset{-A}} \E{Y_B^2} + \sum_{B,C \in \pset{-A} : B \neq C} \E{Y_AY_B} \\
        &\geq \sum_{B \in \pset{-A}} \E{Y_B^2} -  \sum_{B,C \in \pset{-A} : B \neq C} \cz{V_A, V_B}\sqrt{\E{Y_A^2}} \sqrt{\E{Y_B^2}} \\
        &\geq \sum_{B \in \pset{-A}} \E{Y_B^2} -  \sum_{B,C \in \pset{-A} : B \neq C} \cd{\hLs{A}, \hLs{B}}\sqrt{\E{Y_A^2}} \sqrt{\E{Y_B^2}}
    \end{align*}
    where the first inequality is achieved thanks to \lmeref{lme:relaAngles}. Denote $E_A = \lrpar{\sqrt{\E{Y_B^2}}}_{B \in \pset{-A}}$ and notice that
    \begin{align*}
        \sum_{B \in \pset{-A}} \E{Y_B^2} -  \sum_{B,C \in \pset{-A} : B \neq C} \cd{\hLs{A}, \hLs{B}}\sqrt{\E{Y_A^2}} \sqrt{\E{Y_B^2}} = E_A^\top \DeltA{A} E_A
    \end{align*}
    Denote $\lambda_A$ the smallest eigenvalue of $\DeltA{A}$, and notice that if \Assuref{assu:a2} holds, $\DeltA{A}$ is definite positive and $\lambda_A >0$. Thus, one has that
    \begin{align*}
        E_A^\top \DeltA{A} E_A &\geq \lambda_A E_A^\top E_A \\
        &= \lambda_A \sum_{B \in \pset{-A}} \E{Y_A^2}.
    \end{align*}
    Hence, one has that
    \begin{align*}
        \sqrt{\E{\lrpar{\sum_{B \in \pset{-A}} Y_B}^2}} &\geq \sqrt{\lambda_A \sum_{B \in \pset{-A}} \E{Y_A^2}} \\
        &\geq \sqrt{\frac{\lambda_A}{2^d -1}} \sum_{B \in \pset{-A}} \sqrt{\E{Y_A^2}}
    \end{align*}
    where the last inequality is achieved using Jensen's inequality. Hence, one has that, for any $Y \in H_A$
    $$\sqrt{\E{T_A(Y)^2}} \geq \sqrt{\frac{\lambda_A}{2^d -1}} \sum_{B \in \pset{-A}} \sqrt{\E{Y_A^2}}$$
    where $\sqrt{\frac{\lambda_A}{2^d -1}} >0$, and $\sum_{B \in \pset{A}} \sqrt{\E{Y_A^2}}$ is the norm of $Y$ on $H_A$. Hence, by the closed ranged operator theorem (\see \cite[Theorem 2.5]{Abramovich2002}),
    $$\Ran{T_A} = \bigplus_{B \in \pset{-A}} V_B\text{ is closed in } \hLs{A}.$$
\end{proof}

We can now proceed with the proof of \thmref{thm:thm}.

\begin{proof}[Proof of \thmref{thm:thm}]
    The proof is done in two steps. First, we prove by induction that, $\forall A \in \pset{D}$
    $$\hLs{A} = \bigplus_{C \in \pset{A}} V_C,$$
    and then we show that this sum of subspaces is direct. The induction is done on the number of inputs up to $d$.

    \textbf{Statement.}~
    Let $n=1, \dots, d-1$. We will show that if for every non-empty $C \in \pset{D}$, $C$ such that $|C| = n$, one has that
        $$\hLs{C} = \bigplus_{Z \in \pset{C}} V_Z ~\text{ where }~ V_C = \lrbra{\bigplus_{Z \in \pset{-C}} V_Z}^\perpa{C}.$$
    In this case, it follows that for every $A \in \pset{D}$ such that $|A| = n+1$,
        $$\hLs{A} = \bigplus_{C \in \pset{A}} V_C ~\text{ where }~ V_A = \lrbra{\bigplus_{Z \in \pset{-A}} V_Z}^\perpa{A}.$$
    The induction is made on the passage from $n$ to $n+1$.

    \textbf{Base Case.}~
    We start for $n=1$. For any $i \in D$, denote $V_i = \lrbra{V_{\emptyset}}^\perpa{i}$, and notice that since $V_{\emptyset}$ is closed in $\hLs{i}$
    $$\hLs{i} = V_{\emptyset} \oplus V_i.$$
    and notice that $\forall i \in D$, 
    $$V_i = \lrbra{\hLs{\emptyset}}^\perpa{i} \subseteq \lrbra{\hLs{\emptyset}}^\perp,$$
    by \lmeref{lme:perpas}.
    
    Next, consider the case where $n=2$. Notice from the previous step that for any $i,j \in D$ such that $i \neq j$, notice that $\hLs{i \cap j} = \hLs{\emptyset}$, and thus one has that
    $$V_i \subset \lrbra{\hLs{\emptyset}}^\perp \text{ and } V_j \subset \lrbra{\hLs{\emptyset}}^\perp.$$
    Hence, assuming that \Assuref{assu:a1} hold, from \lmeref{lme:step1}, one can conclude that, for any $i,j \in D$ such that $i \neq j$,
    $$\cz{V_i, V_j} \leq \cd{\hLs{i}, \hLs{j}}.$$
    Now, let $A \in \pset{D}$ such that $|A|=2$, and denote $A=\{i,j\}$, and notice that, under \Assuref{assu:a2}, by \lmeref{lme:step2}, one has that
    $$V_{\emptyset } + V_i + V_j \text{ is closed in } \hLs{A}.$$
    Hence, let
    $$V_A = \lrbra{V_{\emptyset } + V_i + V_j}^\perpa{A},$$
    and notice that
    $$\hLs{A} = \lrbra{V_{\emptyset} + V_i + V_j} \oplus V_A.$$
    Since $i$ and $j$ have been chosen arbitrarily, this holds for any $A \in \pset{D}$ such that $|A|=2$.

    \textbf{Induction.}~
    Suppose that, for every $B \in \pset{D}$ such that $|B|=n$, one has that
    $$\hLs{B} = \bigplus_{Z \in \pset{B}} V_Z, \text{ where } V_B = \lrbra{\bigplus_{Z \in \pset{-B}} V_Z}^\perpa{B}.$$
    
    Let $A \in \pset{D}$ such that $|A| = n+1$. Notice then that, for any non-empty $B,C \in \pset{-A}$, since $B \cap C \in \pset{-B} \cap \pset{-C}$, that
    $$\hLs{B \cap C} = \bigplus_{Z \in \pset{B \cap C}} V_Z,$$
    is necessarily contained of $\bigplus_{Z \in \pset{-B}} V_Z$ and of $\bigplus_{Z \in \pset{-C}} V_Z$. Thus, one has that
    $$V_B = \lrbra{\bigplus_{Z \in \pset{-B}} V_Z}^\perpa{B} \subset \lrbra{\bigplus_{Z \in \pset{-B}} V_Z}^\perp \subset \lrbra{\hLs{B \cap C}}^\perp.$$
    and analogously
    $$V_C \subset \lrbra{\hLs{B \cap C}}^\perp.$$
    Hence, assuming that \Assuref{assu:a1} hold, from \lmeref{lme:step1}, one can conclude that, for every non-empty $B,C \in \pset{-A}$ such that $B \neq C$,
    $$\cz{V_B, V_C} \leq \cd{\hLs{B}, \hLs{C}},$$
    which, under \Assuref{assu:a2} and thanks to \lmeref{lme:step2}, implies that $\bigplus_{Z \in \pset{-A}} V_Z$ is closed in $\hLs{A}.$
    Denote $V_A = \lrbra{\bigplus_{Z \in \pset{-A}} V_Z}^\perpa{A}$, and notice that
    $$\hLs{A} = \lrbra{\bigplus_{Z \in \pset{-A}} V_Z} \oplus V_A = \bigplus_{Z \in \pset{A}} V_Z.$$
    Since $A$ has been taken arbitrarily, this holds for any $A \in \pset{D}$ such that $|A|=n$.

    Now, we show that this decomposition is direct. Let $A \in \pset{D}$, and notice that for any non-empty $\forall B \in \pset{A}$, $V_B \perp \hLs{\emptyset}$, meaning that any $f(X_B) \in V_B$ is centered. Next, notice that the principal $\lrpar{2^{|A|} \times 2^{|A|}}$ submatrix of $\Delta$, indexed by the elements of $\pset{A}$ and denoted $\Delta_A$, is also definite-positive, and hence its smallest eigenvalue $\lambda_A$ is positive. Notice further that for any $Y \in \hLs{A}$, by definition, one has that:
    $$Y = \sum_{B \in \pset{A}} Y_B, \quad \text{ where } Y_B \in V_B.$$
    Now, suppose that $Y=0$ a.s., which is equivalent to $\E{Y}=0$ and $\E{Y^2}=0$. However, under Assumptions~\ref{assu:a1} and \ref{assu:a2}, notice that
    \begin{align*}
        \E{Y^2} &= \E{\lrpar{\sum_{B \in \pset{A}} Y_B}^2} \\
        &= \sum_{B \in \pset{A}} \E{Y_B^2} + \sum_{B,C \in \pset{A} : B \neq C} \E{Y_BY_C} \\
        &\geq \sum_{B \in \pset{A}} \E{Y_B^2} - \sum_{B,C \in \pset{A} : B \neq C} \cz{V_B, V_C} \sqrt{\E{Y_B^2}}\sqrt{ \E{Y_C^2}} \\
        &\geq \sum_{B \in \pset{A}} \E{Y_B^2} - \sum_{B,C \in \pset{A} : B \neq C} \cd{\hLs{B}, \hLs{C}} \sqrt{\E{Y_B^2}}\sqrt{ \E{Y_C^2}}
    \end{align*}
    Let $E_A = \lrpar{\sqrt{\E{Y_B^2}}}_{B \in \pset{A}}^\top$ and notice that
    \begin{equation*}
        \E{Y^2} \geq E_A^\top \Delta_A E_A \geq \lambda_A E_A^\top E_A = \lambda_A \sum_{B \in \pset{A}} \E{Y_B^2}
    \end{equation*}
    since $\Delta_A$ is definite positive, and $\lambda_A >0$ is its smallest eigenvalue. Thus, one has that if $\E{Y^2} = 0$, then necessarily
    $$\sum_{B \in \pset{A}} \E{Y_B^2} = 0,$$
    and since this is a sum of positive elements, $\forall B \in \pset{A}$, $\E{Y_B^2}=0$. In addition to the fact that each $Y_B$ is centered, it is equivalent to every the summand $Y_B$ being equal to $0$ a.s. Hence,
    $$Y = 0 \text{ a.s. } \implies \forall B \in \pset{D}, \quad Y_B = 0 \text{ a.s. }$$
    which ultimately proves that
    $$\hLs{A} = \bigoplus_{B \in \pset{A}} V_B.$$
\end{proof}

\section{Some observations} \label{sec:Obs}
Two main observations are discussed in this section. First, the ``hierarchical orthogonality'' structure that naturally arises from the orthocanonical decomposition developed above. Then, we leverage canonical projections due to the direct-sum decomposition of \thmref{thm:thm} to characterize the elements of the decomposition and we present some of their properties.

\subsection{Hierarchical orthogonality}

The set of subspaces $(V_A)_{A \in \pset{D}}$ benefit of a particular orthogonality structure, namely \emph{hierarchical orthogonality}, reminiscent of the one described in \cite{Chastaing2012}.
\begin{prop}[Hierarchical orthogonality]
    We place ourselves in the framework of \thmref{thm:thm}. For any $A \in \pset{D}$, and any $B \subset A$,
    $$V_A \perp V_B.$$
    \label{prop:hierOrth}
\end{prop}
\begin{proof}[Proof of \propref{prop:hierOrth}]
    It is a direct consequence of the definition of $V_A$.
\end{proof}

This particular structure can be illustrated using the Boolean lattice \citep{Davey2002}. In order to formally differentiate between the structurally hierarchical subspaces and those that are not necessarily orthogonal, two different sets related to this structure are introduced. For any $A \in \pset{D}$, the first one is the set of \emph{comparables} (\ie the elements of $\pset{D}$ that are subsets of $A$ or such that $A$ is a subset of), denoted $\calC_A = \pset{A} \cup \lrcubra{B \in \pset{D} : A \subseteq D}$, and notice that, for any $B \in \calC_{A}$, $V_B \perp V_A$. Then, we define the set of \emph{uncomparables} of $A$ as $\calU_A = \pset{D} \setminus \calC_A$, and notice that, in general, for every $B \in \calU_A$, $V_A$ is not necessarily orthogonal to $V_B$. And notice additionally that, for any $A \in \pset{D}$, $\pset{D} = \calC_A \cup \calU_A$.

\begin{rmk}
    It is important to note that the hierarchical orthogonality of the subspaces $\lrpar{V_A}_{A \in \pset{D}}$ is a consequence of the \textit{choice of inductively choosing orthogonal complements} in \thmref{thm:thm}. Other complements, \ie not necessarily orthogonal, could have been chosen, leading to a different structure. This is why the output decomposition in \cororef{coro:uniqDecomp} is called ``orthocanonical''. A different choice of complements could lead to a different decomposition.
\end{rmk}

\subsection{Canonical oblique projections and their properties} \label{sec:oblProj}
Two different projectors onto the subspaces $V_A$ ($A \in \pset{D}$), can be defined. Let $A$ be any element of $\pset{D}$, and denote by $P_A$ the orthogonal projector onto $V_A$. Additionally, for every $A \in \pset{D}$, denote $W_A = \bigoplus_{B \in \pset{D}: B \neq A} V_B$, and the operators
\begin{align*}
    Q_A : \hLs{X} &\rightarrow \hLs{X} \\
    G(X) = \sum_{B \in \pset{D}} G_B(X_B) &\mapsto G_A(X_A)
\end{align*}
and notice that $Q_A$ is the projector onto $V_A$ parallel to $W_A$. This operator is well-defined thanks to the direct-sum decomposition of \thmref{thm:thm} (\see \cite[Theorem 3.4]{Rakic2018}). It is important to note that the operators $P_A(\cdot)$ and $Q_A(\cdot)$ are both projectors onto $V_A$, but their nullspaces differ.

Now, denote by $\Ebb_A$ the orthogonal projector onto $\hLs{A}$, and notice that it is the conditional expectation operator given $X_A$ (\see \cite[Chapter 8]{Kallenberg2021}). Additionally, denote $\widecheck{W}_A = \bigoplus_{B \in \pset{D}, B \not \in \pset{A}} V_B$ and the operator
\begin{align*}
    \Mbb_A : \hLs{X} &\rightarrow \hLs{X} \\
    G(X) = \sum_{B \in \pset{D}} G_B(X_B) &\mapsto \sum_{B \in \pset{A}} G_B(X_B).
\end{align*}
Notice that $\Mbb_A$ is the projection onto $\Ran{\Mbb_A} = \hLs{A}$ parallel to $\Ker{\Mbb_A} = \widecheck{W}_A$. Hence, the operators $\condE{A}{\cdot}$ and $\Mbb_A\lrbra{\cdot}$ are two projections onto $\hLs{A}$, but with different nullspaces as well.
\begin{prop}[Annihilating property]
    We place ourselves in the framework of \thmref{thm:thm} and \cororef{coro:uniqDecomp}. For any $A \in \pset{D}$ and any $B \subset A$,
    $$P_B\lrpar{Q_A\lrpar{G(X)}} = P_B\lrpar{G_A(X_A)} = 0.$$
    \label{prop:anniProp}
\end{prop}
\begin{proof}[Proof of \propref{prop:anniProp}]
    From \propref{prop:hierOrth}, for every $B \subset A$, one has that $V_B \perp V_A$, and thus $G_A(X_A) \in V_A \subset V_B^\perp$.
\end{proof}
\propref{prop:anniProp} is a particular consequence of the hierarchical orthogonality structure, known as the \emph{annihilating property} (\seg \cite[Lemma 1]{Hooker2007} or \cite{Kuo2009}).

It is possible to find a formula to link the oblique projections $\lrpar{Q_A}_{A \in \pset{D}}$ to the oblique projections $\lrpar{\Mbb_A}_{A \in \pset{D}}$ by using combinatorial arguments.
\begin{prop}[Formula for oblique projections]
    We place ourselves in the framework of \thmref{thm:thm} and \cororef{coro:uniqDecomp}. One has that, for any $G(X) \in \hLs{X}$, and for any $A \in \pset{D}$
    $$Q_A(G(X)) = \sum_{B \in \pset{A}} (-1)^{|A| - |B|} \Mbb_A\lrbra{G(X)}.$$
    \label{prop:mobius}
\end{prop}
\begin{proof}[Proof of \propref{prop:mobius}]
    By definition of $\Mbb_A$, one has that
    $$\forall A \in \pset{D}, \quad \Mbb_A\lrbra{G(X)} = \sum_{B \in \pset{A}} Q_A(G(X)),$$
    which, thanks to Rota's generalization of the Möbius inversion formula \citep{Rota1964, ilidrissi2023}, is equivalent to
    $$\forall A \in \pset{D}, \quad Q_A(G(X)) = \sum_{B \in \pset{A}} (-1)^{|A| - |B|} \Mbb_A\lrbra{G(X)}.$$
\end{proof}

\subsection{The particular case of mutual independence} \label{sec:mutIndep}
For a random element $X = (X_1, \dots, X_d)$, mutual independence can be defined \wrt the independence of their generated $\sigma$-algebras. More precisely, $X$ is said to be mutually independent if \citep[Definition 3.0.1]{Malliavin1995}
$$\forall A \in \pset{D}, \quad \sigma_A \indep \sigma_{D \setminus A} \iff \cz{\hLz{A}, \hLz{D \setminus A}} = 0.$$

\begin{prop}
    Let $X$ be a vector of random elements. If $X$ is mutually independent, then \Assuref{assu:a1} hold.
    \label{prop:indepA1}
\end{prop}
\begin{proof}[Proof of \propref{prop:indepA1}]
    From \cite{Malliavin1995}, note that for two $\sigma$-algebras $\calB_1$ and $\calB_2$,
    $$\calB_1 \indep \calB_2 \quad \implies \quad \calB_1 \cap \calB_2 = \sigma_{\emptyset}.$$
    Suppose that \Assuref{assu:a1} does not hold. Hence, in particular, for any $A \in \pset{D}$,
    $$\sigma_A \cap \sigma_{D \setminus A} \neq \sigma_{\emptyset}.$$
    It implies that $\sigma_A$ and $\sigma_{D \setminus A}$ cannot be independent. Hence, since this holds for any $A \in \pset{D}$, $X$ cannot be mutually independent. The result is proven by taking the opposite implication.
\end{proof}

\begin{prop}
    Let $X$ be a vector of random elements and suppose that \Assuref{assu:a1} holds. $X$ is mutually independent if and only if $\forall A, B \in \pset{D}$, $A \neq B$,
    $$\cd{\hLs{A}, \hLs{B}}=0.$$
    \label{prop:mutFried}
\end{prop}
\begin{proof}[Proof of \propref{prop:mutFried}]
    Notice that, in general, if $B \subset A$, then $\cd{\hLs{A}, \hLs{B}}$ is necessarily equal to zero. Thus, we focus on the case where $A\cap B = C \not \in \{A,B\}$.
    Now, suppose that for any $A,B \in \pset{D}$, $\cd{\hLs{A}, \hLs{B}}= 0$. Hence, in particular, under \Assuref{assu:a1}, notice that for every $A \in \pset{D}$
    \begin{align*}
        \cd{\hLs{A}, \hLs{D \setminus A}}  &= \cz{\hLs{A} \cap \hLs{\emptyset}^\perp, \hLs{D \setminus A}\cap \hLs{\emptyset}^\perp} \\
        &= \cz{\hLz{A}, \hLz{D \setminus A}}.
    \end{align*}
    Thus, for every $A \in \pset{D}$,
    $$\cz{\hLz{A}, \hLz{D \setminus A}} = 0 ~ \iff ~ \sigma_A \indep \sigma_{D \setminus A},$$
    which is equivalent to $X$ being mutually independent.
    Now, suppose that $X$ is mutually independent, which implies that, for any $A,B \in \pset{D}$, with $A\cap B = C \not \in \{A,B\}$,
    $$\Ebb_{A} \circ \Ebb_{B} = \Ebb_{B} \circ \Ebb_{A} = \Ebb_{C},$$
    Thus, the orthogonal projections onto $\hLs{A}$ and $\hLs{B}$ commute, which is equivalent to (\see \cite{Kallenberg2021}) 
    $$\cd{\hLs{A}, \hLs{B}}= 0.$$    
\end{proof}

\begin{coro}
    Let $X$ be a vector of random elements and suppose that \Assuref{assu:a1} holds. $X$ is mutually independent if and only if its Feshchenko matrix $\Delta$ is the $\lrpar{2^d \times 2^d}$ identity matrix.
    \label{coro:deltaIdentite}
\end{coro}
\begin{proof}[Proof of \cororef{coro:deltaIdentite}]
    It is a direct consequence of \propref{prop:mutFried}, by definition of $\Delta$.
\end{proof}

Hence, if the inputs are mutually independent, both \Assuref{assu:a1} and \Assuref{assu:a2} hold and lead to the very particular case of $\Delta$ being the identity matrix. One has the following result when it comes to the resulting decomposition of $\hLs{X}$.
\begin{prop}
    Let $X$ be random inputs and suppose that \Assuref{assu:a1} holds. $X$ is mutually independent if and only if
    $$\forall A,B \in \pset{D}, B \neq A \quad V_A \perp V_B.$$
    \label{prop:fullOrtho}
\end{prop}
\begin{proof}[Proof of \propref{prop:fullOrtho}]
    Notice that, in general, if \Assuref{assu:a1} hold, one has that for any $A,B \in \pset{D}$, $B \neq A$
    $$\cz{V_A, V_B} \leq \cd{\hLs{A}, \hLs{B}}.$$
    Note that, from \propref{prop:indepA1}, \Assuref{assu:a1} holds for a mutually independent $X$. Moreover, notice from \propref{prop:mutFried} that $X$ is mutually independent if and only if, $\forall A,B \in \pset{D}$, $A \neq B$, $\cd{\hLs{A}, \hLs{B}} = 0$, thus necessarily $\cz{V_A, V_B} =0$, which is equivalent to $V_A \perp V_B$.
\end{proof}

\propref{prop:fullOrtho} is, in fact, equivalent to the Hoeffding functional decomposition for mutually independent inputs, which can be seen as a very particular case of \thmref{thm:thm} where $X$ admits a Feshchenko matrix equal to the identity, and provided \Assuref{assu:a1} holds. In this case, the subspaces $V_A$ are all pairwise orthogonal, and the \emph{canonical projectors are orthogonal projectors}. \thmref{thm:thm} can thus be seen as a generalization of Hoeffding's classical decomposition to inputs with Feshchenko matrices \emph{that differ from the identity}.

\section{Sensitivity analyses with dependent inputs}\label{sec:decomps}

This section is dedicated to the \emph{decompositions of quantities of interest (QoIs)} ensuing from the orthocanonical decomposition of \cororef{coro:uniqDecomp}. We focus on two QoIs: an evaluation (\ie, prediction) of a model and its variance.

\subsection{Orthocanonical evaluation decomposition}
For $\omega \in \Omega$, denote $x= X(\omega ) \in E$ a realisation of $X$. Subsequently, denote $G(x) \in \R$ the evaluation on $x$ of a random output $G(X) \in \hLs{X}$. In the XAI literature, ``explanation'' methods aim at decomposing $G(x)$ into parts for which each input is responsible \citep{Barredo2020}. They often rely on cooperative game theory, particularly on the Shapley values \citep{Shapley1951}, an allocation with seemingly reasonable properties \citep{Lundberg2017}. However, allocations can be understood as aggregations of coalitional decompositions \citep{ilidrissi2023}, which can be trivially chosen. However, \thmref{thm:thm}, and in particular \cororef{coro:uniqDecomp} offers an orthocanonical approach.

\begin{dfi}[Orthocanonical decomposition of an evaluation]
    Let $X = (X_1, \dots, X_d)$ be a vector of random elements, let $G(X)$ be in $\hLs{X}$ and assume that Assumptions~\ref{assu:a1} and \ref{assu:a2} hold. For any $\omega \in \Omega$, denote $x=X(\omega)$. The orthocanonical coalitional decomposition of the evaluation $G(x)$ is defined as
    $$G(x) = \sum_{A \in \pset{D}} G_A(x_A),$$
    where $x_A = X_A(\omega)$, and 
    $$G_A(x_A) = Q_A\lrpar{G(x)} = \sum_{B \in \pset{A}} (-1)^{|A|-|B|}\Mbb_B\lrbra{G(x)},$$
    where $Q_A$ is the projection onto $V_A$ parallel to $W_A$ and $\Mbb_A$ is the projection onto $\hLs{A}$ parallel to $\widecheck{W}_A$. 
\end{dfi}

The usual coalitional decomposition of choice, even for dependent inputs, relies on choosing conditional expectations (also known as ``conditional Shapley values'') \citep{Lundberg2017}. However, the following results show that this choice entails a canonical decomposition if and only if the inputs are mutually independent.
\begin{prop} \label{prop:condiShap}
     Let $X = (X_1,\dots, X_d)$ be a vector of random elements, let $G(X)$ be in $\hLs{X}$, and assume that Assumptions~\ref{assu:a1} and \ref{assu:a2} hold. Then,
     $$G_A(x_A) = \sum_{B \in \pset{A}} (-1)^{|A|-|B|}\Ebb_B\lrbra{G(x)}, \quad \forall A \in \pset{D}$$
     if and only if $X$ is mutually independent.
\end{prop}
\begin{proof}[Proof of \propref{prop:condiShap}]
    First, notice that $\Mbb_A = \Ebb_A$ if and only if $\widecheck{W}_A$ is the orthogonal complement of $\hLs{A}$. One can notice that, $\widecheck{W}_A$ is a complement of $\hLs{A}$ in $\hLs{X}$, and from \propref{prop:fullOrtho}, one has that
    $$\hLs{A} = \bigoplus_{B \in \pset{A}} V_B \perp \widecheck{W}_A = \bigoplus_{B \in \pset{D}, B \not \in \pset{A}} V_B,$$
    holds for every $A \in \pset{D}$ if and only if $X$ is mutually independent. In this case, $\widecheck{W}_A$ is an orthogonal complement of $\hLs{A}$, and by uniqueness, $\widecheck{W}_A = \hLs{A}^\perp$, and thus $\Mbb_A = \Ebb_A$.
\end{proof}

\subsection{Variance decomposition}
Let $G(X)$ be a random output and denote its variance by
$$\V{G(X)} = \E{\lrpar{G(X) - \E{G(X)}}^2} = \E{G(X)^2} - \E{G(X)}^2.$$

We propose two ways to approach the problem of decomposing $\V{G(X)}$. The \emph{orthocanonical variance decomposition} relies on the decomposition of $G(X)$ offered by \cororef{coro:uniqDecomp}. In contrast, the \emph{organic variance decomposition} aims at defining and disentangling \emph{pure interaction effects} from \emph{dependence effects}. 

\subsubsection{Orthocanonical variance decomposition}
In light of \cororef{coro:uniqDecomp}, the orthocanonical variance decomposition of $G(X)$ is rather intuitive. It relies on the following rationale:
\begin{align*}
    \V{G(X)} &= \Cov{G(X), G(X)} \\
    &= \sum_{A \in \pset{D}} \Cov{G_A(X_A), G(X)} \\
    &= \sum_{A \in \pset{D}} \lrbra{\V{G_A(X_A)} + \sum_{B \in \calU_A} \Cov{G_A(X_A), G_B(X_B)}}.
\end{align*}
reminiscent of the ``covariance decomposition'' \citep{Stone1994, Chastaing2012, Hart2018, DaVeiga2021}. Two indices arise from this decomposition. 

\begin{dfi}[Orthocanonical variance decomposition]
    We place ourselves in the framework of \thmref{thm:thm}. For any $A \in \pset{D}$, define
    \begin{itemize}
        \item The \emph{structural contribution of $X_A$ to $G(X)$} by
        $$S_A^U \eqdef \V{G_A(X_A)};$$
        \item The \emph{correlative contribution of $X_A$ to $G(X)$} by
        $$S_A^C \eqdef \sum_{B \in \calU_A} \Cov{G_A(X_A), G_B(X_B)}.$$
    \end{itemize} 
\end{dfi}
The orthocanonical decomposition of $\V{G(X)}$ is suitable in practice if the dependence structure of $X$ is assumed to be inherent in the modeling of the studied phenomenon. These indices can be expressed using canonical oblique projections (\see \Secref{sec:oblProj}).
\begin{prop}\label{prop:projCorrel}
    We place ourselves in the framework of \thmref{thm:thm}. Then, for any $A \in \pset{D}$
    $$S_A^C = \sum_{B \in \pset{A}} (-1)^{|A|-|B|} \Cov{\Mbb_B\lrbra{G(X)}, \lrpar{I-\Mbb_A}\lrbra{G(X)}}.$$
\end{prop}
\begin{proof}[Proof of \propref{prop:projCorrel}]
    First, recall that, for any $A \in \pset{D}$,
    $$G_A(X_A) = \sum_{B \in \pset{A}} (-1)^{|A|-|B|} \Mbb_B\lrbra{G(X)},$$
    and hence,
    \begin{align*}
        \sum_{B \in \pset{A}} (-1)^{|A|-|B|} \Cov{\Mbb_B\lrbra{G(X)}, \lrpar{I-\Mbb_A}\lrbra{G(X)}} &= \Cov{G_A(X_A),  \lrpar{I-\Mbb_A}\lrbra{G(X)}} \\
        &= \sum_{B \in \pset{D} : B \not \in \pset{A}} \Cov{G_A(X_A),  G_B(X_B)}.
    \end{align*}
    However, notice that $\calU_A \subset \pset{D} \setminus \pset{A}$, and that, for any $B \in \pset{D} \setminus \pset{A}$ with $B \not \in \calU_A$,
    $$\Cov{G_A(X_A), G_B(X_B)} = 0,$$
    and hence,
    \begin{align*}
        \sum_{B \in \pset{A}} (-1)^{|A|-|B|} \Cov{\Mbb_B\lrbra{G(X)}, \lrpar{I-\Mbb_A}\lrbra{G(X)}} &= \sum_{B \in \pset{D} : B \not \in \pset{A}} \Cov{G_A(X_A),  G_B(X_B)} \\
        &= \sum_{B \in \calU_A} \Cov{G_A(X_A),  G_B(X_B)} \\
        &= S_A^C.
    \end{align*}
\end{proof}

\begin{prop}\label{prop:projStructu}
     We place ourselves in the framework of \thmref{thm:thm}. Then, for any $A \in \pset{D}$
    $$S_A^U = \sum_{B \in \pset{A}} (-1)^{|A|-|B|} \lrbra{\V{\Mbb_B\lrbra{G(X)}} - \Cov{\Mbb_B\lrbra{G(X)}, \lrpar{I-\Mbb_A}\lrbra{G(X)}}}.$$
\end{prop}
\begin{proof}[Proof of \propref{prop:projStructu}]
    First, recall that
    $$\Mbb_A\lrbra{G(X)} = \sum_{B \in \pset{A}} G_B(X_B).$$
    Thus, 
    \begin{align*}
        \V{\Mbb_A\lrbra{G(X)}} &= \V{\sum_{B \in \pset{A}} G_B(X_B)} = \sum_{B \in \pset{A}} \V{G_B(X_B)} + \sum_{C \in \calU_A} \Cov{G_B(X_B), G_C(X_C)} \\
        &= \sum_{B \in \pset{A}} S_B^U + S_B^C
    \end{align*}
    which is equivalent to
    $$\V{\Mbb_A\lrbra{G(X)}} - \sum_{B \in \pset{A}} S_B^C= \sum_{B \in \pset{A}} S_B^U.$$
    However, notice that, $\forall A \in \pset{D}$,
    $$\sum_{B \in \pset{A}} S_B^C = \Cov{\Mbb_A\lrbra{G(X}), \lrpar{I-\Mbb_A}\lrbra{G(X)}},$$
    and thus, $\forall A \in \pset{D}$,
    $$\V{\Mbb_A\lrbra{G(X)}} - \Cov{\Mbb_A\lrbra{G(X)}, \lrpar{I-\Mbb_A}\lrbra{G(X)}} = \sum_{B \in \pset{D}} S_B^U.$$
    Using Rota's generalization of the Möbius inversion formula applied to the power-set, it yields that, $\forall A \in \pset{D}$,
    $$ S_A^U = \sum_{B \in \pset{A}} (-1)^{|A|-|B|}\lrbra{\V{\Mbb_B\lrbra{G(X)}} - \Cov{\Mbb_B\lrbra{G(X)}, \lrpar{I-\Mbb_A}\lrbra{G(X)}}}.$$
\end{proof}

\subsubsection{Organic variance decomposition}
The goal of the \emph{organic variance decomposition} is to separate ``pure interaction effects'' to ``dependence effects''. Denote by $\tildeX=(\tildeX_1, \dots, \tildeX_d)$ the mutually independent version of $X$, \ie the $E$-valued random element with the same univariate marginals as $X$, but such that $\tildeX$ is mutually independent. Suppose additionally that $G(X) \in \hLs{X}$ and $G\lrpar{\tildeX} \in \hLs{\tildeX}$.

\begin{dfi}[Pure interaction effect]
    We place ourselves in the framework of \thmref{thm:thm}. For any $A \in \pset{D}$, the \emph{pure interaction effects of $X_A$ on $G(X)$} are defined as
    $$S_A = \frac{\V{\tildeG_A\lrpar{\tildeX_A}}}{\V{G\lrpar{\tildeX}}}\V{G(X)}.$$
\end{dfi}
These indices are the Sobol' indices of $G\lrpar{\tildeX}$ \citep{Sobol2001}, which are known in the literature as quantifying pure interaction \citep{DaVeiga2021}. They are analogous to the work of \cite{Mara2015}.
\begin{rmk}
    In certain situations, when $X$ is from a certain family of random vectors, it is possible to find a simple mapping $T : E \rightarrow E$ such that
    $$\tildeX = T(X).$$
    In particular, if $P_X$ is in the family of elliptical distribution, it amounts to performing a Nataf transform of the inputs \citep{Lebrun2009a, Lebrun2009b}.
\end{rmk}

One desirability criterion can be brought forward to define dependence effects: the set of indices must all be equal to zero if and only if $X$ is mutually independent.
\begin{lme}\label{lme:oblOrth}
    We place ourselves in the framework of \thmref{thm:thm}. Let $G(X) \in \hLs{X}$. Then, 
    $$Q_A\lrpar{G(X)} = P_A\lrpar{G(X)} \text{ a.s.}, \quad \forall A \in \pset{D} \quad \iff \quad X \text{ is mutually independent.}$$
\end{lme}
\begin{proof}[Proof of \lmeref{lme:oblOrth}]
    First, suppose that $X$ is mutually independent. By \propref{prop:fullOrtho}, one has that
    $$\forall A,B \in \pset{D}, B \neq A \quad V_A \perp V_B,$$
    which entails that
    $$V_A \perp W_A = \bigplus_{B \in \pset{D} : B \neq A} V_B.$$
    However, notice that $W_A$ still complements $V_A$ in $\hLs{X}$. Furthermore, by unicity of the orthogonal complement, one has that $W_A = V_A^\perp$. Thus,
    $$\Ran{Q_A} = V_A, \quad \Ker{Q_A} = V_A^\perp,$$
    and thus $Q_A = P_A$, leading to
    $$Q_A\lrpar{G(X)} = P_A\lrpar{G(X)} \text{ a.s.}$$
    Now, suppose that for any $A \in \pset{D}$, $Q_A\lrpar{G(X)} = P_A\lrpar{G(X)}$ a.s. Hence, it implies that 
    $$\forall A \in \pset{D}, G_A(X_A) = P_A\lrpar{G(X)}.$$
    which implies that $P_A = Q_A$, since the above equation defines the operator $Q_A$. Thus, $P_A$ and $Q_A$ must share the same ranges and nullspaces. In particular,
    $$\Ker{Q_A} = \Ker{P_A} = V_A^\perp,$$
    implying that $W_A = V_A^\perp$, which leads to
    $$V_A \perp W_A, \forall A \in \pset{D} \quad \iff \quad V_A \perp V_B, \forall A,B \in \pset{D}, B \neq A.$$    
    Finally, thanks to \propref{prop:fullOrtho}, notice that this is equivalent to $X$ being mutually independent.
\end{proof}
Taking the distance between the canonical and orthogonal projections naturally ensues as a measure of dependence. 
\begin{dfi}[Dependence effects]
    We place ourselves in the framework of \thmref{thm:thm}. For any $A \in \pset{D}$, define the \emph{dependence effect of $X_A$ on $G(X)$} as
    $$S_A^D = \V{Q_A(G(X)) - P_A(G(X))} = \E{\lrpar{Q_A(G(X)) - P_A(G(X))}^2}.$$
\end{dfi}
Moreover, this distance does respect the desirability criteria for dependence effects.
\begin{prop}\label{prop:DepIndep}
    $$\forall A \in \pset{D}, ~ S_A^D=0  \quad \iff \quad X \text{ is mutually independent.}$$
\end{prop}
\begin{proof}[Proof of \propref{prop:DepIndep}]
    This is a direct consequence of \lmeref{lme:oblOrth}.
\end{proof}

\section{Analytical formulas for two Bernoulli inputs} \label{sec:illustration}

Let $X=(X_1, X_2)$ where both inputs follow a Bernoulli distribution (here, $E = \{0,1\}^2$) with success probability $q_1$ and $q_2$ respectively. The joint law of $X$ can be fully expressed using three free parameters: $q_1$, $q_2$, and $\rho = \E{X_1X_2}$. More precisely, one has that:
\begin{equation*}
p_{00} = 1-q_1-q_2+\rho, ~p_{01} = q_2 - \rho, ~ p_{10} = q_1 - \rho, ~ p_{11} = \rho
\end{equation*}
where, for $i,j \in \{0,1\}$, one denotes $p_{ij} = \prob\lrpar{\lrcubra{X_1 = i} \cap \lrcubra{X_2=j}}$. Denote the $(4 \times 4)$ diagonal matrix $P = \textrm{diag}(p_{00},p_{01},p_{10},p_{11})$. 

Any function $G : \{0,1\}^2 \rightarrow \R$ can be represented as a vector in $\R^4$, where each element represents a value that $G$ can take \wrt the values taken by $X$. For $i,j \in \{0,1\}$, denote $G_{ij} = G(i,j)$, and thus $G = \lrpar{G_{00}, G_{01}, G_{10}, G_{11}}^\top$ where each $G_{ij}$ can be observed with probability $p_{ij}$. In this particular case, one can analytically compute the decomposition of $G$ related to \thmref{thm:thm}. It can be performed by finding the suitable unit-norm vectors in $\R^4$
\begin{equation}
    v_{\emptyset} = \begin{pmatrix}c\\c\\c\\c \end{pmatrix}, v_1 = \begin{pmatrix}g_0\\g_0\\g_1\\g_1 \end{pmatrix}, v_2 = \begin{pmatrix}h_0\\h_1\\h_0\\h_1 \end{pmatrix}, v_{12}=\begin{pmatrix}
        k_{00}\\
        k_{01}\\
        k_{10}\\
        k_{11}
    \end{pmatrix}
    \text{ such that }
    \begin{cases}
        v_{\emptyset}^\top P v_1 = 0 \\
        v_{\emptyset}^\top P v_2 = 0 \\
        v_{\emptyset}^\top P v_{12} = 0 \\
        v_{12}^\top P v_1 = 0 \\
        v_{12}^\top P v_2 = 0 \\
    \end{cases} 
    \text{ and } 
    \begin{cases}
        v_{\emptyset}^\top P v_{\emptyset} = 1\\
        v_{1}^\top P v_1=1\\
        v_{2}^\top P v_2=1\\
        v_{12}^\top P v_{12}=1\\
    \end{cases}
\end{equation}
which results in a system of nine equations with nine real unknown parameters. Given these vectors, one has that any function $G$ can be written as
$$ G = ev_{\emptyset} + \alpha v_1 + \beta v_2 + \delta v_{12},$$
resulting in four additional equations with four unknown parameters. These 13 equations and 13 parameters can be solved analytically using, e.g., the symbolic programming package \texttt{sympy} \citep{sympy}. We refer the interested reader to the accompanying \href{https://github.com/milidris/GeneralizedAnova}{GitHub repository}\footnote{\href{https://github.com/milidris/GeneralizedAnova}{https://github.com/milidris/GeneralizedAnova}} for the analytical formulas obtained for this decomposition, as well as the analytical formulas of the indices introduced in \Secref{sec:decomps}.

\section{Discussion and future work} \label{sec:conclu}

The first main challenge towards adopting the indices presented in \Secref{sec:decomps} is estimation. While many methods exist to estimate conditional expectations, many of these schemes rely on the variational problem offered by Hilbert's projection theorem, \ie characterizing orthogonal projections as a distance-minimizing problem. A first exploration path would be to express oblique projections as a distance-minimizing optimization problem under constraints. A second approach would be to take advantage of the particular expression of oblique projections (\seg \cite{Afriat1957, Corach2005}). A final strategy would be to find suitable bases for each $(V_A)_{A \in \pset{D}}$ to project $G(X)$ onto, analogously to \cite{Chastaing2012}. Non-orthogonal polynomial bases could be a great start to study this problem.

The second main challenge is understanding the extent of the proposed methodology. We believe it is a step towards a more global treatment of dependencies in (non-linear) multivariate statistics. Our framework offers a (somewhat surprisingly) linear approach to possibly highly non-linear models.

Finally, the role of the Boolean lattice is essential in our analysis and offers a path toward studying different algebraic structures. Rota's generalization of the Möbius inversion formula \citep{Rota1964, CombiRota2012} is referenced many times in our developments. In fact, Rota's result is much more general and does not only apply to Boolean lattices (\ie powersets) and paves the way for more complex analysis. For instance, the relationships between the inputs may differ: hierarchical structures (\eg to represent physical causality) or the presence of trigger variables \citep{Pelamatti2023}. Such considerations could result in a different intrinsic algebraic structure that still remain partially ordered.

\section*{Acknowledgements}

Support from the ANR-3IA Artificial and Natural Intelligence Toulouse Institute is gratefully acknowledged. 

\bibliographystyle{plain}
\bibliography{bib/references}

\end{document}

%% file: macros.tex
\usepackage[OT1]{fontenc}
\usepackage{amsmath,amsfonts, amssymb,amsthm}
\usepackage[numbers]{natbib}
\usepackage[colorlinks,citecolor=blue,urlcolor=blue]{hyperref}
\usepackage[left=3cm,right=3cm,top=3cm,bottom=3cm]{geometry}
\usepackage{bm}
\usepackage{graphicx}
\usepackage{dsfont}
\usepackage{tabularx}
\usepackage{stmaryrd}
\usepackage[svgnames,dvipsnames]{xcolor}
\usepackage{listings}
\usepackage{mathtools}
\usepackage{multirow}

\newcommand{\lrpar}[1]{\left( #1 \right)}
\newcommand{\lrbra}[1]{\left[ #1 \right]}
\newcommand{\lrcubra}[1]{\left\{ #1 \right\}}

\newcommand{\ie}{\text{i.e.,}~}
\newcommand{\eg}{\text{e.g.,}~}
\newcommand{\seg}{\text{see, e.g.,}~}
\newcommand{\see}{\text{see}~}

\newcommand{\wrt}{\text{w.r.t.}~}

\newcommand{\Eqref}[1]{Eq.~(\ref{#1})}

\newcommand{\Secref}[1]{Section~\ref{#1}}

\newcommand{\Assuref}[1]{Assumption~\ref{#1}}
\newcommand{\thmref}[1]{Theorem~\ref{#1}}
\newcommand{\cororef}[1]{Corollary~\ref{#1}}
\newcommand{\defref}[1]{Definition~\ref{#1}}
\newcommand{\propref}[1]{Proposition~\ref{#1}}
\newcommand{\lmeref}[1]{Lemma~\ref{#1}}

\newcommand{\E}[1]{\mathbb{E}\left[ #1 \right]}

\newcommand{\condE}[2]{\mathbb{E}_{#1}\lrbra{#2}}
\newcommand{\V}[1]{\mathbb{V}\left( #1 \right)}

\newcommand{\Cov}[1]{\textrm{Cov}\left( #1 \right)}
\newcommand{\Prob}[1]{\mathbb{P} \left( #1 \right)}
\newcommand{\prob}{\mathbb{P}}
\newcommand{\indep}{\perp \!\!\! \perp}

\newcommand{\eqdef}{:=}
\newcommand{\iprod}[1]{\left\langle #1 \right\rangle}
\newcommand{\norm}[1]{\left\lVert #1 \right \rVert}
\newcommand{\absval}[1]{\left| #1 \right|}
\newcommand{\Ran}[1]{\textrm{Ran}\left( #1 \right)}
\newcommand{\Ker}[1]{\textrm{Ker}\left( #1 \right)}

\newcommand{\calB}{\mathcal{B}}
\newcommand{\calE}{\mathcal{E}}
\newcommand{\calC}{\mathcal{C}}

\newcommand{\calM}{\mathcal{M}}

\newcommand{\calF}{\mathcal{F}}
\newcommand{\calH}{\mathcal{H}}
\newcommand{\calU}{\mathcal{U}}

\newcommand{\tildeX}{\widetilde{X}}
\newcommand{\tildeG}{\widetilde{G}}

\newcommand{\R}{\mathbb{R}}
\newcommand{\Ebb}{\mathbb{E}}
\newcommand{\Mbb}{\mathbb{M}}

\newcommand{\hL}{\mathbb{L}^2}

\newcommand{\hLs}[1]{\hL\lrpar{\sigma_{#1}}}
\newcommand{\hLz}[1]{\hL_0\lrpar{\sigma_{#1}}}
\newcommand{\cz}[1]{c_0\lrpar{#1}}
\newcommand{\cd}[1]{c\lrpar{#1}}
\newcommand{\pset}[1]{\mathcal{P}_{#1}}
\newcommand{\perpa}[1]{{\perp_{#1}}}
\newcommand{\DeltA}[1]{\Delta\vert_{#1}}

\makeatletter
\newcommand{\bigplus}{%
  \DOTSB\mathop{\mathpalette\mattos@bigplus\relax}\slimits@
}
\newcommand\mattos@bigplus[2]{%
  \vcenter{\hbox{%
    \sbox\z@{$#1\sum$}%
    \resizebox{!}{0.9\dimexpr\ht\z@+\dp\z@}{\raisebox{\depth}{$\m@th#1+$}}%
  }}%
  \vphantom{\sum}%
}
\makeatother

\theoremstyle{plain}
\newtheorem{lme}{Lemma}
\newtheorem{coro}{Corollary}
\newtheorem{prop}{Proposition}
\newtheorem{thm}{Theorem}
\newtheorem{dfi}{Definition}

\theoremstyle{remark}
\newtheorem{assu}{Assumption}
\newtheorem{rmk}{Remark}

\makeatletter
\def\ps@pprintTitle{%
\let\@oddhead\@empty
\let\@evenhead\@empty
\def\@oddfoot{}%
\let\@evenfoot\@oddfoot}
\makeatother

\definecolor{backcolour}{rgb}{0.95,0.95,0.92}
\DeclareFontFamily{U}{mathx}{}
\DeclareFontShape{U}{mathx}{m}{n}{<-> mathx10}{}
\DeclareSymbolFont{mathx}{U}{mathx}{m}{n}
\DeclareMathAccent{\widecheck}{0}{mathx}{"71}